\documentclass[oneside,12pt]{article}
\usepackage[english]{babel}
\RequirePackage{latexsym} \RequirePackage{amsthm}
\RequirePackage{amsmath}

\RequirePackage{amssymb} \RequirePackage{makeidx}

\usepackage{enumerate}

\theoremstyle{definition}
\newtheorem{defin}{Definition}[section]
\newtheorem{prop}[defin]{Proposition}
\newtheorem{theorem}[defin]{Theorem}

\newtheorem{lemma}[defin]{Lemma}
\newtheorem{cor}[defin]{Corollary}
\newtheorem*{thm}{Main Theorem}
\newtheorem*{remark}{Remark}
\theoremstyle{plain} \newtheorem{exa}[defin]{Example}

\RequirePackage{fancyhdr} 

\fancypagestyle{plain}{%
    \fancyhead{} %
    }

\usepackage{color}

\makeatletter
\def\cleardoublepage{\clearpage\if@twoside \ifodd\c@page\else
  \hbox{}
  \vspace*{\fill}
  \begin{center}
  \end{center}
  \vspace{\fill}
  \thispagestyle{empty}
  \newpage
  \if@twocolumn\hbox{}\newpage\fi\fi\fi}
\makeatother

\usepackage{pgf,tikz}
\usetikzlibrary{arrows}

\makeatletter
\def\cleardoublepage{\clearpage\if@twoside \ifodd\c@page\else
  \hbox{}
  \vspace*{\fill}
  \begin{center}
  \end{center}
  \vspace{\fill}
  \thispagestyle{empty}
  \newpage
  \if@twocolumn\hbox{}\newpage\fi\fi\fi}
\makeatother

\newcommand{\rank}{\mathop{\operator@font rank}}

\makeatother

\numberwithin{equation}{section}

\makeatletter
\renewcommand\section{\@startsection {section}{1}{\z@}%
                                   {-3.5ex \@plus -1ex \@minus -.2ex}%
                                   {2.3ex \@plus.2ex}%
                                   {\normalfont\large\bfseries}}
\renewcommand\subsection{\@startsection{subsection}{2}{\z@}%
                                     {-3.25ex\@plus -1ex \@minus -.2ex}%
                                     {1.5ex \@plus .2ex}%
                                     {\normalfont\bfseries}}
\makeatother

\makeindex

\begin{document}
\title{{Asymmetry of Outer Space of a free product}}	

\author{Dionysios Syrigos}
\date{\today}
\maketitle

\begin{abstract}
For every free product decomposition $G = G_{1} \ast ...\ast G_{q} \ast F_{r}$ of a group of finite Kurosh rank $G$, where $F_r$ is a finitely generated free group, we can associate some (relative) outer space $\mathcal{O}$. We study the asymmetry of the Lipschitz metric $d_R$ on the (relative) Outer space $\mathcal{O}$.
More specifically, we generalise the construction of Algom-Kfir and Bestvina, introducing an (asymmetric) Finsler norm $\|\cdot\|^{L}$ that induces $d_R$.
Let's denote by $Out(G, \mathcal{O})$ the outer automorphisms of $G$ that preserve the set of conjugacy classes of $G_i$'s. Then there is an $Out(G, \mathcal{O})$-invariant function $\Psi : \mathcal{O} \rightarrow \mathbb{R}$ such that when $\| \cdot \|^{L}$ is corrected by $d \Psi$, the resulting norm is quasisymmetric.
As an application, we prove that if we restrict $d_R$ to the $\epsilon$-thick
part of the relative Outer space for some $\epsilon >0$, is quasi-symmetric . Finally, we generalise for IWIP automorphisms of a free product a theorem of Handel and Mosher, which states that there is a uniform bound which depends only on the group, on the ratio of the relative expansion factors of any IWIP $\phi \in Out(F_n)$ and its inverse.
\end{abstract}

	\tableofcontents
\section{Introduction}
Outer Space is a very well studied space, which can be used to study the group of outer automorphisms $Out(F_n)$ of a finitely generated free group $F_n$. There are a lot of combinatorial and topological methods to study the space. However, Francaviglia and Martino in \cite{FM1} introduced a natural asymmetric Lipschitz metric $d_R$ on $CV_n$. We could define also a symmetric version of this metric, but the non-symmetric one is geodesic and seems natural in terms of studying the dynamics of free group automorphisms. Recently, this metric theory and the resulting
geometric point of view have been used extensively to study the Outer Space. As a consequence, we can get many new results, as well as more elegant new proofs of older results, for example see: \cite{Algom Kfir 2}, \cite{Algom Kfir}, \cite{Bestvina}, \cite{FM2} and \cite{Hamenstandt}.\\
On the other hand, Guirardel and Levitt in \cite{GF} constructed an outer space relative to any free product decomposition of a group $G = G_{1} \ast ...\ast G_{q} \ast F_{r} $ of finite Kurosh rank. There are a lot of analogies between the classical and the general Outer Space. Firstly, Francaviglia and Martino in \cite{FM}  introduced and studied the Lipschitz metric for the general case. In the same paper, they proved as an application, the existence of train track representatives for (relative) IWIP automorphisms. Moreover, many well known constructions and theorems of the free case can be generalised in the general case (for example, see \cite{CT}, \cite{HL}, \cite{H1}, \cite{H2}, \cite{H3}, \cite{Sykiotis}  and \cite{Syrigos}). This is a motivation to study further analogies, and in particular here we study the asymmetric metric $d_R$.\\
In this paper, we generalise the construction of Algom-Kfir and Bestvina in \cite{Algom Kfir} following closely their approach, as we introduce an asymmetric Finsler norm on the tangent space of the relative Outer space that induces the asymmetric Lipschitz metric. We also show how to correct this norm to make it quasi-symmetric. Our main result explains the lack of quasi-symmetry in terms of a certain function and more specifically:

\begin{thm}
There is an $Out(G, \mathcal{O})$-invariant continuous, piecewise smooth function $\Psi : \mathcal{O} \rightarrow \mathbb{R}$ and constants $A, B > 0$ (depending only on the numbers $r,q$)
such that for every $T,S \in \mathcal{O}$ we have $d(T, S) \leq A \cdot d(S, T) + B \cdot [\Psi(T) - \Psi(S)]$.
\end{thm}

As an application, we prove that if we restrict the asymmetric metric $d_R$ to the $\epsilon$-thick
	part of the relative Outer space for $\epsilon >0$, which is the subspace of $\mathcal{O}$ of the points for which all hyperbolic elements have length bounded below by $\epsilon$, is quasi-symmetric (actually, we just need the multiplicative constant). Finally, we generalise a theorem of Handel and Mosher (see \cite{Handel + Mosher}), that there is a uniform bound, which depends only on the numbers $r$ and $q$, on the ratio of the relative expansion factors of any IWIP $\phi \in Out(G, \mathcal{O})$ and its inverse. Since any automorphism $\phi \in Out(G)$ is irreducible relative to some appropriate space $\mathcal{O}$, we can apply the general theorem to get a result for the expansion factors of any automorphism $\phi \in Out(F_n)$ and its inverse, as in the general theorem of $\cite{Handel + Mosher}$.

\textbf{Acknowledgements:} I wish to thank my advisor Armando Martino for his help, suggestions and corrections.
	
\section{Preliminaries}
\subsection{Kurosh rank and $\mathbb{R}$-trees}
Let's suppose that $G$ is a group which splits as a finite free product $G = H_{1} \ast ...\ast H_{r} \ast F_{n}$, where every $H_i$ is non-trivial, not isomorphic to $\mathbb{Z}$ and freely indecomposable. We say that such a group has \textit{finite Kurosh rank} and such a decomposition is called \textit{Gruskho decomposition}.  Note that the $G_i$'s are unique, up to conjugacy and the ranks $n,r$ are well defined. The number $r+n$ is called the Kurosh rank of $G$. Finally, every f.g. group admits a splitting as above (by the theorem of Grushko). We are interested only for groups which have finite Kurosh rank.\\
Now for a group $G$ of finite Kurosh rank, we fix an arbitary (non-trivial) free product decomposition $G = H_{1} \ast ...\ast H_{r} \ast F_{n}$, i.e without assuming that each $H_i$ is not isomorphic to $\mathbb{Z}$ or freely indecomposable. Note that these groups admit co-compact actions on $\mathbb{R}$-trees (and vice-versa).\\
More specifically, for a simplicial tree $T$ (not necessarily locally compact), we denote by
$V(T)$ and $E(T)$ the set of vertices and edges of $T$, respectively. We put also a metric on the tree $T$, by assigning a positive length to each edge and we can think $T$ as a $\mathbb{R}$-tree.  
Now, for $x, y \in T$, we denote by $[x, y]$ the unique path from $x$ to $y$, and for any reduced path $p$ in $T$ we denote by $\ell_T (p)$ the length of $p$ in $T$ which is defined by summing the lengths of the edges that $p$ crosses.\\
We consider only isometric actions of the group $G$ on $\mathbb{R}$-trees and, more specifically, we say that $T$ is a \textit{$G$-tree,} if it is a simplicial metric tree $ (T, d_T )$, where $G$ acts simplicially on $T$ (sending vertices to vertices and edges to edges) and for all $g \in G, e \in E(T) $ we have that $e$ and $ge$ are isometric. Moreover, we suppose that every $G$-action is \textit{minimal}, which means that there is no $G$-invariant proper subtree.\\
Now let's fix a $G$-tree $T$. An element $g \in G$ is called \textit{hyperbolic}, if it doesn't fix any points of $T$. Any hyperbolic element $g$ of $G$ acts by translation on a subtree of $T$ homeomorphic to the real line, which is called the axis of $g$ and denoted by $axis_T (g)$. The \textit{translation
length} of $g$ is the distance that $g$ translates its axis. The action of $G$ on $T$
defines a length function denoted by
\begin{equation*}
\ell_T : G \rightarrow R, \ell_T (g) : = \inf \limits_{x \in T} d_T (x, gx).
\end{equation*}
In this context, the infimum is always minimum and we say that $g \in G$ is hyperbolic if and only if $\ell_T (g) > 0$. Otherwise, $g$ is called \textit{elliptic} and it fixes a point of $T$.
Finally, if $g$ is hyperbolic, we can find some $v \in axis_T(g)$ s.t. the unique reduced path from $v$ to $gv$ has length exactly $\ell_T (g)$. Sometimes, the segment $[v, gv]$ (or even the loop $\alpha$ on which $[v,gv]$ projects to  $\Gamma = G/T$) is called the \textit{ period of the axis}.
For more details about $\mathbb{R}$-trees, see \cite{Culler Morgan}.
 \subsection{Outer Space and The simplex of Metrics}
Let's fix an arbitrarily free product decomposition $G = G_{1} \ast ...\ast G_{r} \ast F_{n}$ of a group $G$ of finite Kurosh rank.
Note that it is useful that we can also apply the theory in the case that $G$ is free, and the $G_i$'s are certain free factors of $G$ (relative free case).\\
Following \cite{GF}, we define an outer space $\mathcal{O} = \mathcal{O}(G, (G_i)^{r}_{i=1}, F_n)$ relative to this free product decomposition (relative outer space).
\begin{defin}\label{Outer Space}
 An element $T$ of the outer space $\mathcal{O}$ can be thought as simplicial metric $G$-tree, up to $G$-equivariant homothety.
Moreover, we require that:
\begin{itemize}
	\item The edge and the tripod stabilisers are trivial.
	\item There are finitely many orbits of vertices with non-trivial stabiliser and more precisely for every $G_i$, $i = 1,..., r$  there is exactly one vertex $v_i$ with stabiliser $G_i$ (all the vertices in the orbits of $v_i$'s are called \textit{non-free vertices}).
	\item  All other vertices have trivial stabiliser (and we call them \textit{free vertices}).
\end{itemize}
\end{defin}
 The mimimality implies that we have finitely many orbits of edges for every tree $T$ and we denote by $E_1(T)$ the finite set which contains exactly one edge of each orbit. Also, for convenience we normalise the length of edges and we suppose that the sum of the lengths of edges in $E_1(T)$ is $1$.\\
 Note that by a remark of \cite{FM}, the hyperbolic elements of $T \in \mathcal{O}$ depends only on the space $\mathcal{O}$ and we denote them by $Hyp(\mathcal{O})$.\\
On the other hand, for a $G$-tree $T$ as above, we can consider a lot of different metrics $\ell$ s.t. $(T, \ell) \in \mathcal{O}$. More specifically, we say that a $G$-invariant function $\ell : E(T) \rightarrow [0,1]$ is a \textit{metric} (relative to $\mathcal{O}$) on $T$, if there is no hyperbolic element $g \in G$ in $\mathcal{O}$ s.t. $\ell_T(g) = 0$.\\
We denote by $\Sigma _ {T}$ the set of all metrics in $T$.
The space $\Sigma_T$ of all metrics $\ell$ on $T$ is a "simplex with missing faces", where the
missing faces correspond to metrics that vanish on a $G$-subgraph that contains the axes of hyperbolic elements. Therefore in that case $(T, \ell)$ is not an element of $\mathcal{O}$.\\
Alternatively, we could define $\mathcal{O}$ as the disjoint union of the simplices $\Sigma_T$, where $T$ varies over all the $G$-trees $T$ which satisfy the assumptions of the Definition \ref{Outer Space}.\\
We would like to define a natural action of $Out(G)$ on $\mathcal{O}$, but this is not possible since it not always the case that the automorphisms of $G$ preserve the structure of the trees, as they may not preserve the conjugacy classes of the $G_i$'s. However, we can describe here the action of a specific subgroup of $Out(G)$ (namely, the automorphisms that preserve the decomposition or, equivalently, the structure of the trees) on $\mathcal{O}$.\\
Let $Aut(G, \mathcal{O})$ be the subgroup of $Aut(G)$  that preserve the set of conjugacy classes of the $G_i$ 's. Equivalently, $\phi \in Aut(G)$ belongs to $Aut(G,\mathcal{O})$ iff $\phi(G_i)$ is conjugate to one of the $G_j$ 's (in general, $i$ is different to $j$). The group $Aut(G,\mathcal{O})$ admits a natural action on a simplicial tree by "changing the action", i.e. for $\phi \in Aut(G, \mathcal{O})$ and $T \in \mathcal{O}$, we define $\phi(T)$ to be the element with the same underlying tree with $T$, the same metric but the action is given by $g*x = \phi(g)x$ (where the action in the right hand side is the action of the $G$-tree $T$). As $Inn(G)$ acts on $\mathcal{O}$ trivially, there is a induced action of $Out (G,\mathcal{O}) = Aut(G,\mathcal{O})/ Inn(G)$ on $\mathcal{O}$. Note also that in the case of the Grushko decomposition we have $Out(G) = Out(G,\mathcal{O})$.\\
\subsection{Tangent spaces}
For every $\ell \in \Sigma_{T} $, we define the tangent space 
\begin{equation*}
T_{\ell} (\Sigma_{T}) = \Big \{ \tau : E(T) \rightarrow \mathbb{R} | \sum \limits_{ e \in E(T) } \tau(e) = 0 \Big \}.
\end{equation*}
Since the tangent space does not depend on the metric, for every two metrics $\ell, \ell '$ the natural identification between $T_{\ell} (\Sigma_{T})$ and $T_{\ell '} (\Sigma_{T})$, implies that the total tangent space can be written as $T (\Sigma_{T}) \cong \Sigma_{T} \times \mathbb{R} ^{N-1}$ where $N$ is the number of edges  of $\Sigma_{T}$. \\
\begin{defin}
A tangent vector $\tau \in T_{\ell} (\Sigma_{T})$ is \textit{integrable} (relative to $\ell$), if $\tau(e) < 0$ implies that $\ell(e) > 0$ for all $e \in E(\Sigma_{T})$, i.e. it is not possible to find an edge $e$ with $\tau(e) < 0 $ and $\ell(e) = 0$.
\end{defin}
Note that if $\tau$ is integrable, then for all sufficiently small $t \geq 0$  we have that $\ell + t \tau \in \Sigma_{T}$.\\
As a consequence, we can define $\tau (p)$ for any reduced path $p$ in $T$, as $\sum\limits_{e} \tau(e)$ where $e$ varies all over the edges that $p$ crosses, counted with multiplicity.
Therefore if $g \in Hyp(\mathcal{O})$ and $L_g$ is the period of the axis of $g$, we can define $\tau(g) : = \tau(L_g)$. 

\subsection{Lipschitz metric and Optimal maps}
In this section, we follow \cite{FM}. Let $A, B \in \mathcal{O}$ be two elements of the outer space and let's denote by $\ell_A$, $\ell_B$ the corresponding translation functions of $A$ and $B$, respectively.
Here we define the (right) stretching factor as:
\begin{equation*}
\Lambda_R(A, B) :=  \sup \limits _ {g \in Hyp({\mathcal{O}})} \frac{\ell_B(g)}{\ell_A (g)}
\end{equation*}
and the (right) asymmetric pseudo-distance as:
\begin{equation*}
d_R(A,B) = d(A,B) : = \log \Lambda_R(A, B) 
\end{equation*}
In the case where $r=2$ and $n=0$, we have just one tree with exactly one orbit of edges. Therefore the metric vanishes. However, in any other case the metric is not symmetric and in fact is not even quasi-symmetric. If $n \geq 2$, we can adjust the counter- examples of the free case in order to work in the general case as well. We will give an examples for the case where $r=n=1$.
\begin{exa}
Suppose that $r=n=1$ and so is of the form $G = G_1 \ast \mathbb{Z} $, where $G_1$ is any group of finite Kurosh rank. Then we have two simplices (marked trees, if we forget the metric) of $G$-trees and let's denote them by $T,S$ s.t. $G/T$ is a loop with a non free vertex and $G/S$ a loop with one edge attached connecting the loop and the non-free vertex which has valence $1$. Now the unique representative of edges of $T$ has length 1 while we give length $\epsilon$ in the edge corresponding to the loop of $G / S$ and $1-\epsilon$ to the other and let's denote this metric tree by $S_\epsilon \in \mathcal{O}$.\\
Now all the hyperbolic elements in $T$ have length $1$, while in $S_\epsilon$ there are hyperbolic elements of length $\epsilon$ and some others with length $2-\epsilon$. Therefore choosing $\epsilon$ sufficiently small, we can see that $d_R(T,S_{\epsilon}) = 2-\epsilon$, while $d_R(S_{\epsilon},T) =  \frac{1}{\epsilon}  \rightarrow \infty$.
\end{exa}
Let's recall the definition of \cite{FM} and some useful properties. We say that a map $f : A \rightarrow B$, where $A,B \in \mathcal{O}$, is an $\mathcal{O}$- map, if it is a $G$-equivariant, Lipschitz continuous, surjective function. One interesting property is the following:

\begin{lemma}
	For every pair $A,B \in \mathcal{O}$; there exists an $\mathcal{O}$-map $f : A \rightarrow B$.
	Moreover, any two $\mathcal{O}$-maps from $A$ to $B$ coincide on the non-free vertices.
\end{lemma}
In addition, it can be proved that for every $A,B$ there is an $\mathcal{O}$-map $f$ which realises the distance between them, which means that the Lipschitz constant of $Lip(f)$ is exactly $d(A,B)$. These maps are called \textit{optimal}. In particular, for every IWIP automorphism $\phi \in Out(G, \mathcal{O})$ relative to $\mathcal{O}$, there is an optimal train track representative $f : T \rightarrow \phi(T)$ of $\phi$ that stretches every edge by a specific number, which is called the \textit{expansion factor} of $\phi$ (relative to $\mathcal{O}$).\\
Finally, we list some useful properties of the metric:
\begin{prop}(Francaviglia and Martino, \cite{FM})
	\begin{enumerate}
	\item For every $A,B \in \mathcal{O}$ there is an optimal map $f : A \rightarrow B$ with $ Lip(f) = \inf \{ Lip(h) | h \mbox{ is an } \mathcal{O} \mbox{ -map from } A \mbox{ to } B \} $
\item	 $d(A, B) \geq 0$ with equality only if $A = B$.
	\item $ d(A, C) \leq d(A, B) + d(B, C)$ for all $A,B,C \in \mathcal{O}$ .
	\item  $d$ is a geodesic metric. Moreover, a path that realises the distance $d(A,B)$ for every $A,B \in \mathcal{O}$ can be chosen to be piecewise linear, and even linear in each simplex.
\item $Out(G, \mathcal{O})$ acts on $\mathcal{O}$ by isometries.
\end{enumerate}
\end{prop}

\subsection{Candidates and more}

\begin{figure}[t]
	
	\includegraphics[scale=0.8]{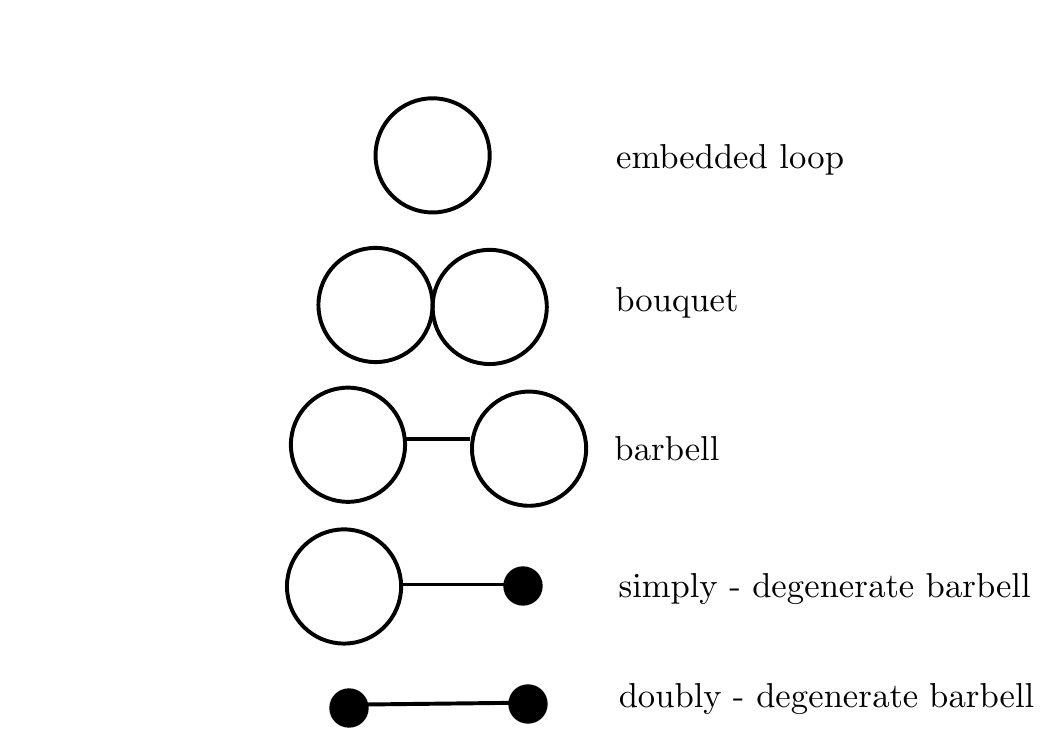}
	\centering
	\caption{Projections of candidates}
\end{figure}

\begin{defin}\label{Candidates}
 An element $g \in G $ is a candidate in $T$, if it is hyperbolic in $T$ and, denoting by $axis_T (g)$ its
 axis in $T$, there exists $v \in axis_T (g)$ such that the segment $[v, gv]$ projects to a loop $\alpha$ in the quotient graph $\Gamma := G / T$ which is either (see also Figure 1)
 \begin{enumerate}
 \item an embedded loop, or
 \item a bouquet of two circles in $\Gamma$, i.e. $\alpha = \alpha_1 \alpha_2$, where $\alpha_1$ and $\alpha_2$ are embedded circles in $\Gamma$ which meet in a single point, or
 \item a barbell graph, i.e. $\alpha = \alpha_1 \beta \alpha_2 \bar{ \beta}$, where $\alpha_1$ and $\alpha_2$ are embedded circles in $\Gamma$ that do not meet, and $\beta$ is an embedded path in $\Gamma$ that meets $\alpha_1 $ and $\alpha_2$ only at their  origin (and we denote by $\bar{ \beta}$  the path $\beta$ crossed in the opposite direction), or
 \item a simply-degenerate barbell, i.e. $\alpha$ is of the form $\alpha \beta \bar{ \beta}$  , where $\alpha$ is an embedded loop in $\Gamma$ and $\beta$ is an embedded path in $\Gamma$, with two distinct endpoints, which meets $\alpha$ only at its origin, and whose terminal endpoint is a non-free vertex in $\Gamma$, or
 \item a doubly-degenerate barbell, i.e. $\alpha$ is of the form $\beta \bar{ \beta}$, where $ \beta$ is an embedded path in $\Gamma$ whose two distinct endpoints are different non-free vertices.
  \end{enumerate}
\end{defin}

We denote by $C_T$ the set of candidates in $T$.
\begin{defin}
Let $g,g' \in G$ be hyperbolic elements in $\mathcal{O}$, for which $pr(axis_T(g)) = pr(axis_T(g'))$, or in other words they project to the same path to the quotient $\Gamma = G/T$, then we say that they are projectively equivalent in $T$ (or just projectively equivalent).
\end{defin}
\begin{remark}
Note that if $g,g'$ are projectively equivalent in $T$, then for any $\ell \in \Sigma_T$ : $\tau (g) = \tau(g')$, $ \ell(g) = \ell (g')$. Moreover,
there are finitely many projectively inequivalent hyperbolic elements of bounded length. In particular, there are finitely many projectively inequivalent candidates.
\end{remark}
The next proposition shows that the distance is realised on a candidates and it is essential for our arguments. In particular:
\begin{prop}
	For any $A,B \in \mathcal{O}$,
	\begin{equation*}
	d(A, B) =    \max \limits _ {g \in C_A}  \frac{\ell_B(g)}{\ell_A (g)}
	\end{equation*}
\end{prop}

\section{Basic Lemma}
Let's assume that $A, B \in \mathcal{O}$. One main question is that if we change slightly $B$, can we compute the distance $d(A, B)$ using the same candidate of $A$? We will prove that this is possible under some conditions.

\begin{defin}
	A closed convex cone, in a finite dimensional real vector space $V$, is a closed subset $C$ of $V$ such that $v,w \in C$ implies that $tv +sw \in C$ for all $t,s \in [0, \infty )$.
	
\end{defin}
One main example of a closed convex cone is the set of integrable vectors in $T_{\ell} (\Sigma_{T})$.\\\\
\textbf{Notational Convention}: When we restrict our attention to a specific simplex $\Sigma _T$ for a specific $G$-tree $T$ in Outer space, we may identify the point $(T, *, \ell)$, where we denote the $G$-action on $T$ by $*$, by only specifying the metric $\ell$.\\
\\
Firstly, we prove a very useful proposition which states that in a specific case we can use the same candidate which realises the distance.
\begin{prop}\label{Basic Prop}
\begin{enumerate}
\item Let $\tau \in T_{\ell} (\Sigma_{T})$ be an integrable vector. Then there is a candidate $\alpha$ in $\Sigma_{T}$ such that
\begin{equation*}
d(\ell, \ell + t \tau ) = log \frac{ (\ell + t \tau)(\alpha)}{\ell(\alpha)}
\end{equation*}
for all sufficiently small $t \geq 0$, i.e. the same candidate $a$ realises the distance $d(\ell, \ell + t \tau )$ for small $t$. Moreover, $\alpha$ has the property that for any other hyperbolic element $g$, $\frac{\tau(\alpha)}{\ell(\alpha)} \geq \frac{\tau(g)}{\ell(g)}$ .
\item $\lim _{t \rightarrow 0^+} \frac{d(\ell, \ell + t \tau )}{t} = \frac{\tau (\alpha)}{\ell(\alpha)}$, where $\alpha$ is the candidate of item (i).
\item The set of integrable vectors in $T_{\ell} (\Sigma_{T})$ can be written as a finite union of closed convex cones $B_1, B_2, ..., B_M$ such that for any $B_i$, there is a (projective equivalence class of a) candidate $\alpha_i$ that realises the distance $d(\ell, \ell + t \tau )$ for any $\tau \in B_i$ and for all sufficiently small $t \geq 0$.
\end{enumerate}
\end{prop}

\begin{proof}
Let $\alpha$ be a candidate in $T$ that realises $d(\ell, \ell + t \tau )$. This is equivalent to the inequalities :
\begin{equation*}
\frac{ (\ell + t \tau)(\alpha)}{\ell(\alpha)} \geq \frac{ (\ell + t \tau)(g)}{\ell(g)}
\end{equation*}
for all hyperbolic elements $g$ in $\mathcal{O}$. But since the distance can be realised by a candidate, it is enough to consider these inequalities only for the candidates. Moreover, as we have seen we have finitely many classes of projectively inequivalent candidates and so we need finitely many of these inequalities. Let choose a representative of each class and let's denote them by $\alpha _i \in C_T$, $i= 1,..., M $.\\
On the other hand, we can simplify these inequalities to $\frac{\tau(\alpha)}{\ell(\alpha)} \geq \frac{\tau(g)}{\ell(g)}$ when $t >0$.\\
This is a finite system of linear inequalities which determines a closed convex cone $B_i$ associated to each $\alpha_i$ as in (iii) and more specifically
\begin{equation*}
 \tau \in B_i \Longleftrightarrow \frac{\tau(\alpha _i)}{\ell(\alpha _i)} \geq \frac{\tau(a_j)}{\ell(a_j)} , \mbox{ for every } j=1,...,M.
\end{equation*}
The inequalities  do not depend on $t$ and so we have (i), since we can choose the same candidate to realise the distance for all small $t$  and the second part of the statement is evident by the discussion above.\\
Finally, using the item (i), we can divide by $t$ in order to calculate the limit which is straightforward and then the item (ii) follows.
\end{proof}

\section{Norm}
As in the section above, we fix a tree $T$.
We can now define a function in $\Sigma_{T} \times T_{\ell} (\Sigma_{T})$, which is a norm. Therefore, we will have a norm in the tangent space which induces the Lipschitz metric. We fix a metric $\ell \in \Sigma_T$ and we give the next definition: 
\begin{defin}
Let $\tau \in T_{\ell} (\Sigma_{T})$. Then we define:
\begin{equation*}
\| (\ell, \tau) \| ^{L} = \ \sup \Big \{ \frac{\tau(g)}{\ell(g)} \Big | g  \in Hyp ( \mathcal{O})  \Big \} \end{equation*}

\end{defin}
We will prove that we have an (asymmetric) Finsler norm for the Lipschitz metric.

\begin{prop}
\begin{enumerate}
\item If $\tau$ is integrable, then $\| (\ell, \tau) \| ^{L} = \lim \limits _{t \rightarrow 0^+} \frac{d(\ell, \ell + t \tau )}{t}$ 
\item The supremum in the definition is achieved on a candidate of $\Sigma_{T}$.
\item $\| (\ell, \tau) \| ^{L}$ is continuous on $T(\Sigma_{T})$.
\item $\| (\ell, \tau) \| ^{L} \geq 0$ with equality iff $\tau = 0$.
\item $\| (\ell, \tau_1 + \tau_2) \| ^{L} \leq \| (\ell, \tau_1) \| ^{L} + \| (\ell, \tau_2) \| ^{L}$
\item If $c > 0$, then $\| (\ell, c\tau) \| ^{L} = c \| (\ell, \tau) \| ^{L}$
\end{enumerate}

\end{prop}

\begin{proof}
	\begin{enumerate}

\item Since $\tau$ is integrable, we can use \ref{Basic Prop} (iii) and we have that there is some $i$ s.t. $\tau \in B_i$, but then using the item (ii) of the same proposition we get that there exists some candidate $\alpha _i$ with the property
\begin{equation*}
 \lim _{t \rightarrow 0^+} \frac{d(\ell, \ell + t \tau )}{t} = \frac{\tau (\alpha_i)}{\ell(\alpha_i)}\end{equation*} 
Therefore, (i)  of Proposition \ref{Basic Prop} establishes that $\frac{\tau(\alpha_i)}{\ell(\alpha_i)} \geq \frac{\tau(g)}{\ell(g)}$ for any other hyperbolic element $g$ in $\mathcal{O}$, which means that the supremum in the definition can be achieved on the candidate $\alpha_i$.

\item If $\tau$ is integrable, then as we have seen above there is a candidate $\alpha$ with the property $\| (\ell, \tau) \| ^{L} = \frac{\tau(\alpha)}{\ell(\alpha)}$, and so the supremum is realised on some candidate of $T$.\\
Now let $\tau$ be not integrable, which means that there is some edge $e$ s.t. $\ell(e) = 0$ and $\tau(e) < 0$. But we can always find some $\ell '$ which is as close as we want to $\ell$ (which means that $\ell ' (e) = \epsilon$ for small $\epsilon$) so that $\tau$ becomes integrable (relative to $\ell'$). Therefore if the perturbation is sufficiently small, the candidate that works for the pair $(\ell, \tau)$, works for $(\ell ', \tau)$ as well.
 
\item This follows from (ii), since we can replace the sup of the definition by a maximum over a finite set (of projectively inequivalent classes of candidates for graphs in the simplex $\Sigma_T$, i.e. the projections of candidates to $\Gamma = G / T$).

\item If $\tau \neq 0$, then we can produce some $g \in Hyp (\mathcal{O})$ so that $\tau(g) > 0$. We can do this without dependence on $\ell$, so we may assume that (as above, changing $\ell$) $\tau$ is integrable, and then by \ref{Basic Prop}(i) and the item (i) above, we have that for sufficiently small $t$: $0 < d(\ell, \ell +t\tau) = log (1 + t \| (\ell, \tau) \| ^{L})$, which implies that there exists such element $g$.
\\
\par
(v) and (vi) are straightforward, using the definition and the properties of supremum.

\end{enumerate}

	\end{proof}
	
		\begin{figure}[t]
			
			\includegraphics{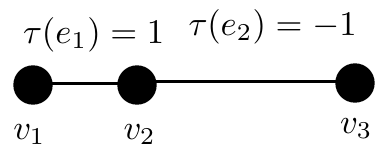}
			\centering
			\caption{The quotient $G/T$}
		\end{figure}
		
		However, as in the free case the norm is not quasi-symmetric. For $n>1$, we can essentially use the examples (adjusted appropriately) of the free case. We give an example if there is no free part.
		\begin{exa}
			For $n=0$ and $r=3$ and so $G = G_1 \ast G_2 \ast G_3$. Let $T$ be the tree with three orbits of non-free vertices which we denote by $v_1,v_2,v_3$ with vertex groups $G_1, G_2, G_3$, respectively. Let's also suppose that there are exactly two orbits of edges and namely we denote by $e_1$ the edge connecting $v_1$ with $v_2$ and by $e_2$ the edge connecting $v_2$ with $v_3$. Moreover, if we assume that $\ell(e_1) =\epsilon , \ell(e_2) = 1- \epsilon$ and $\tau(e_1) = 1, \tau(e_2) = -1$, it's easy to see that there are three types of candidates (see Figure 2). More specifically, the candidates $C_1, C_2, C_3$ that correspond to  $[v_1,v_2]$, $[v_2,v_3]$, $[v_1,v_3]$, respectively. Their lengths are $2\epsilon$ for $C_1$, $ 2(1-\epsilon)$ for $C_2$ and $2$ for $C_3$ and their $\tau$-values are $2$ for $C_1$, $-2$ for $C_2$ and $0$ for $C_3$. Therefore we can see that $\| (\ell, \tau) \| ^{L} = \frac{1}{\epsilon}$, while $ \| (\ell, - \tau) \| ^{L} = \frac{1}{1-\epsilon}$. Now sending $\epsilon$ to $0$, we get  $\| (\ell, \tau) \| ^{L} \rightarrow \infty$, while $ \| (\ell, - \tau) \| ^{L} \rightarrow 1$.
		\end{exa}
\section{Corrected norm}
We would like to define a new norm on $ \Sigma_T \times T_{\ell} (\Sigma_{T})$ which is quasi- symmetric. As in the free case, we have to correct $\| \cdot \| ^{L} $ by adding the directional derivative of a function which is the sum of - $\log$'s of the lengths of candidates. The first problem is that the candidates are not the same, if we change the marking. Therefore, if $n>1$ we need to consider the double covers of trees with the same non-free vertices. However, we have to face one other problem which is the existence of $G_i$'s and which makes the homology insufficient. Instead, we use the "relative abelianisation" of the group $G$ which transforms the free product into a direct product. In order to produce a finite set of representatives our strategy, roughly speaking, is that for the "free part" we can apply the homology with $\mathbb{Z}_2 $ -coefficients and for the $G_i$'s we consider an equivalence relation.\\
For every $T \in \mathcal{O}$, we denote by $P = \pi _1 (G/T)$ the fundamental group of the quotient, which is a graph of groups with vertex groups isomorphic to the $G_i$'s and trivial edges groups. The idea is to consider the "relative abelianisation" of $\pi_1 (G / T)$. More specifically, there is a natural homomorphism $\psi$ from $\pi_1 (G / T)$ to $G_1 \oplus G_{2} \oplus ... \oplus G_r \oplus H_1(\Gamma, \mathbb{Z}_2) $, where $H_1(\Gamma, \mathbb{Z}_2)$ is isomorphic to $\mathbb{Z}_2 ^{n}$. We give the next definition:
\begin{defin}
 We define the set $H(T)$ as the set of equivalence classes in $G_1 \oplus G_{2} \oplus ... \oplus G_r \oplus H_1(\Gamma, \mathbb{Z}_2) $ induced by the equivalence relation:\\ $(g_1, ..., g_r, y_1) \sim (g' _1, ..., g' _r, y_2)$, where $g_i, g' _i \in G_i, y_1, y_2 \in  H_1(\Gamma, \mathbb{Z}_2) $ iff ($g_i \neq 1_{G_i}$ iff $g' _i \neq 1_{G_i} $) and $y_1 = y_2$.\\
 In order to simplify the notation, we exclude the $0$ class (which is the class which can be represented by the element $(1, 1,...,1, 0)$) from $H(T)$.\end{defin}
\textbf{Convention}: If $[g] = [h] \in H(T)$, we abuse the terminology and we say that $g,h$ are in the same homothety class in order to distinguish with the "projectively equivalent hyperbolic elements of $T$" which means that they project in the same path in the quotient $G / T$. Note that there is the possibility that two projectively equivalent elements not to be in the same homothety class.\\
\begin{defin}
If $[x] \in H(T)$, we define the length of the class $[x]$, and we denote it by $\ell ([x])$, to be the infimum of the lengths $\ell(g)$, where $g$ varies over all the hyperbolic elements in the homothety class of $[x]$.
\end{defin}
 We will prove that the infimum is realised and it is actually a minimum, but it might be realised on more than one (projectively) equivalences classes of elements.
\begin{prop}\label{realise}
	\begin{enumerate}
\item For each $[x] \in H(T)$ there are finitely many projectively inequivalent elements $h_1,...,h_k$, so that $\ell(x)$ is realised by some $h_i$ for all $\ell \in \Sigma_T$.
\item Moreover, if for the hyperbolic element $g$ the projection of $axis(g))$ in $\Gamma = G / T$ is an embedded loop or a degenerate barbell (i.e. the cases (i), (iv), (v) of the Definition \ref{Candidates}), then  $g$ realises $\ell([g]), [g] \in H(T)$ and for every other $h$ that realises $\ell ([g])$, we have that $pr(axis(g)) = pr(axis(h))$.

	\end{enumerate}
\end{prop}

\begin{proof}
	\begin{enumerate}

\item

 We claim that if $[x] \in H(T)$ is represented by some hyperbolic element $h$ which realises $\ell([x])$ and $L_h = pr(axis(h))$ crosses the edge $e$ more that once then $L_h$ crosses $e$ exactly twice in opposite directions and $e$ separates the image of $L_h$.\\
In order to see this, we consider two cases. In the first case, suppose  that $ \psi(h) = (h_1, ..., h_r, x_1) \in G_1 \oplus G_{2} \oplus ... \oplus G_r \oplus H_1(\Gamma, \mathbb{Z}_2) $
 and $L_h$ crosses some edge twice in the same direction and without loss we assume that $ L_h = e b_1 e b_2$. But if we "avoid" the edge $e$, the path $b_1 \bar{b}_2$ is homothetic to $L_h$ using $\mathbb{Z}_2$-coefficients. Moreover, we can find a hyperbolic element $h'$ with $\psi(h) \sim \psi(h')$ and $L_{h'} = b_1 \bar{b}_2$ and as a consequence the free parts of $\psi(h) $ and $\psi(h')$ would be homothetic.
 This can be done since the loop $b_1 \bar{b}_2$ contains every $v_i$, where $G_{v_i} = G_i$, for which the corresponding coordinate $h_i$ of $\psi(h)$ is non-trivial. Since $G/T$ is a graph of groups and the elements of its fundamental group have the form $g_0e_1g_1e_2...e_mg_m$, where $e_1,...,e_m$ is a loop based at some point $v_0 \in G_{v_0}$, $g_0 \in G_{v_0}$ and $g_i \in G_{\tau(e_i)}$, we can use $ b_1 \bar{b}_2$ to be our loop and therefore, we can produce a word $h'$ following the loop $b_1 \bar{b}_2$ and only the first time we meet some $G_{v_i}$, as above, we write the letter $h_i$. This will produce a hyperbolic element (since $\psi(h') = \psi(h)$ is not homothetic to $(1,1,...,1,0)$, while every elliptic element is) with "homothetic free part" with $h$ and the rest requested properties.
 As a consequence, we can change $h,L_h$ with the pair $h',L_{h'}$, so that $[h] = [h'] \in H(T)$, $L_{h'}$ contains each non-free vertex $v_i$ for which $h_i \neq 1$, and $L_{h'}$ is strictly shorter than $L_h$. Therefore, $L_h$ cannot contain some edge $e$ twice in same direction.\\
Similarly, for the second case, i.e. if the axis $h$ crosses some edge twice in opposite directions but doesn't separate the image, we can use  the same arguments as in the free case, combining them with the fact that the new path has to contain each non-free vertex $v_i$ for which $h_i \neq 1$, so it is possible to come up with a loop corresponding to a projectively equivalent element with strictly shorter period of its axis.\\
Therefore, if $h$ realises the length of $[h] \in H(T)$, we have that $L_h$ crosses each edge once or it crosses it twice in opposite directions and $e$ separates the image of $L_h$. But there are finitely many such projections of axes of hyperbolic elements and these projections do not depend on the choice of the metric, i.e. there are finitely many projectively inequivalent elements in the homothety class of $[h]$ which realise $\ell ([h]) $.
\item Suppose that $h$ is a hyperbolic element so that $L_h$ is an arc, then it is obvious that for any other element $h'$ s.t. $[h] = [h']$, $L_{h'}$ has to contain the same non-free vertices and actually to cross at least the same edges and therefore if $h'$ realises the length of $[h]$, then $L_h = L_{h'}$.\\
If $h$ is a hyperbolic element so that $L_h$ is an embedded loop, then it is easy to see that for any other $h'$ s.t. $[h] = [h']$, we have that $L_h = L_{h'}$.\\
Finally, if $h$ is a hyperbolic element so that $L_h$ is an embedded loop with an arc attached such that the last vertex of the arc is a non-free vertex, then again we use the fact that for any other $h' $ s.t. $[h] = [h']$, $L_h'$ has to contain the same non-free vertices as $L_h$ and it has to cross all the edges of the arc. As above, it has to cross all the edges of the embedded loop and therefore if $h'$ realises $\ell([h])$ we get that $L_h = L_ {h'}$.
	\end{enumerate}
\end{proof}

The set of linear inequalities $\ell (h_i) \leq \ell (h_j)$ for the set of $h_i$'s in the previous proposition divides the simplex $\Sigma_{T}$ into closed convex subsets $C_1,...,C_k$ s.t. for each $C_i$ there is an $h_i$ with the property $\ell(h_i) \leq \ell(h_j)$ for all $j$. In fact, we define the $C_i$'s by:
\begin{equation*}
\ell \in C_i \Longleftrightarrow  \ell (h_i) \leq \ell (h_j) , \mbox{ for every } j=1,...,k.
\end{equation*}
As a consequence we get the following corollary:
\begin{cor}\label{Closed Convex}
	A simplex $\Sigma_{T}$ is covered by closed convex subsets $C_1,...,C_k$ s.t. for each $x \in H(T)$, there is a hyperbolic element such that $\ell([x]) = \ell (h_j)$ for all $\ell \in C_j$.
\end{cor}
Moreover, we can get as a corollary:

\begin{cor}\label{derivative}
Let choose $[x] \in H(T)$.
For every integrable $\tau \in T_{\ell} (\Sigma_{T})$, there is a $j$ s.t. $\ell, \ell +t \tau \in C_j$ (for all small $t>0$) and the derivative from the right at 0 of $t \rightarrow (\ell +t \tau)([x]) $ is $\tau(h_j)$.\\
Moreover, it equals to $\min \{ \tau(h_i) | [h_i] = [x], h_i \mbox{ realises } \ell([x]) \}$.
\end{cor}
\begin{proof}
Let $[x] \in H(T)$ and (without loss, after reordering) assume that $h_1, ..., h_m$ are the projectively inequivalent hyperbolic elements in the homothety class of $[x]$ which realise $\ell([x])$, i.e. $\ell([x]) = \ell(h_i), i=1,...,m$. Assuming that $\tau(h_1) \leq \tau(h_i), i=1,...,m $, then we have that, for all sufficiently small $t>0$,\begin{equation*}
(\ell +t \tau)(h_1) = \ell(h_1) +t\tau(h_1) \leq \ell(h_i) +t \tau(h_i)  = (\ell +t \tau) (h_i), i = 1,...,m.
\end{equation*}
Therefore $(\ell +t \tau) (h_1)$ realises $(\ell +t \tau) ([x])$. As a consequence, using the previous corollary, $\ell, \ell +t \tau \in C_1$. In this case, it is straightforward that the derivative is exactly $\tau(h_1)$, since actually $(\ell +t \tau)([x]) = (\ell +t \tau)(a_1)$ for all small $t>0$. 
\end{proof}
Now we have to distinguish two cases. The case that free group has rank more than or equal to 2 and the other case. If $n=0$ or $n=1$, then the original arguments don't work, however it turns out that it is actually easier to define the function $\Psi$ without using double covers.

\subsection{$n= 0$ or $n=1$}
In this case, we have that any candidate is an embedded loop or a degenerate barbell. Therefore we are in position to define directly the function $\Psi$ of the main theorem, as there is no need to consider double covers.
\begin{defin}
  	Fixing some $\ell \in \Sigma_{T}$ and $\tau \in T_{\ell} (\Sigma_{T}) $, we define the number:
	\begin{equation}\label{2}
	N(\ell, \tau) = - \sum\limits_{[\alpha] \in H(T)} \frac{\min \tau ([\alpha])}{\ell ([\alpha])}
	\end{equation}
	where minimum is taken over the projectively inequivalent hyperbolic elements $g$ in the homothety class of $[\alpha]$ which realise  $\ell ([\alpha])$.
\end{defin}
We define a new function, which we will prove that it is a norm, by:
\begin{equation}
\| (\ell, \tau) \| ^{N} = \| (\ell, \tau) \| ^{L} + \frac{1}{K+1}N(\ell, \tau)
\end{equation}
where $K$ is the number of summands in \ref{2} (and it depends on $n,r$).
\\
We write $\| \tau \|^{.}$ instead of $\| (\ell, \tau) \| ^{.}$, for simplicity.\\

Define the map $\Psi : \Sigma_{T} \rightarrow \mathbb{R}$ by
\begin{equation}
\Psi(\ell) = - \frac{1}{K+1}  \sum\limits_{[\alpha] \in H(T)} \log \ell ([\alpha])
\end{equation}
 Note that $\Psi$ is smooth on each convex set $C_j$ of the Corollary \ref{Closed Convex}. 
\subsection{Generic Case}
Here we have to consider all the non-trivial "double covers" of $T$ with the same non-free vertices with full stabilisers, $T_i \rightarrow T, i =1,2,...,2^n - 1$. These double covers have quotient with fundamental groups which correspond to the kernels of homomorphisms from $G$ to $\mathbb{Z}_2$, sending every element of each $G_i$ to $0$. We have that $T_i$ is the same metric tree with $T$. Therefore we get that the lifts of $\ell$ and $\tau$ to each $T_i$, and we denote them by $\ell_i$ and $\tau_i$, respectively. Similarly, we can define the space of metrics $\Sigma _{T_i}$ and the corresponding tangent space $T_{\ell_i} (\Sigma _{T_i})$.
A very important lemma which is the reason that we consider the double covers is the following:
\begin{lemma}\label{lifts}
If $\alpha$ is a candidate in $T$, then there exists a double cover $T _i \rightarrow T$, and a lift $\tilde{\alpha}$ of $\alpha$ in $T_i$, so that $\tilde{a}$ is the unique (projective class) element that realises the length of $[ \tilde{\alpha} ]$.
\end{lemma}

\begin{proof}
We will apply Proposition (ii)  \ref{realise} on the appropriate $T_i$.\\
For the case (v) of \ref{realise} i.e. the doubly degenerate barbell, every $T_i$ works.\\ 
For the cases (i), (iv)  of the Definition \ref{Candidates} i.e. the embedded loop and of the simply degenerate barbell, it is enough to suppose that the embedded loop lifts to an embedded loop. Since we can always find such a double cover, the lemma follows.\\
For the cases (ii), (iii) of \ref{realise} i.e. the figure eight and the non-degenerate barbell $L$ with embedded loops $A_1, A_2$, we just have to find a double cover of $T$ on which $A_1, A_2$ does not lift but $L$ lifts. Again, there is always such a double cover.
\end{proof}

\begin{defin}
	Fixing some $\ell \in \Sigma_{T}$ and $\tau \in T_{\ell} (\Sigma_{T}) $, we define the number

\begin{equation}\label{1}
N(\ell, \tau) = - \sum\limits_{T_i} \sum\limits_{[\alpha] \in H(T_i)} \frac{\min \tau_i ([\alpha])}{\ell_i ([\alpha])}
\end{equation}
	where minimum is taken over the projectively inequivalent hyperbolic elements $g$ in the  class of $[\alpha]$ which realise  $\ell_i ([\alpha])$.
\end{defin}
We are now in position to define the new norm by :
\begin{equation}
\| (\ell, \tau) \| ^{N} = \| (\ell, \tau) \| ^{L} + \frac{1}{K+1}N(\ell, \tau)
\end{equation}
where $K$ is the number of summands in \ref{1} (and it depends on  $r,n$).
\\
We write $\| \tau \|^{.}$ instead of $\| (\ell, \tau) \| ^{.}$, for simplicity.\\

Define the map $\Psi : \Sigma_{T} \rightarrow \mathbb{R}$ by
\begin{equation}
\Psi(\ell) = - \frac{1}{K+1} \sum\limits_{T_i} \sum\limits_{[\alpha] \in H(T_i)} log \ell_i ([\alpha])
\end{equation}
where $\ell _i$ is the lift of $\ell$ to $T_i$. Again, note that $\Psi$ is smooth on each convex set $C_j$ of the Corollary \ref{Closed Convex}.

\subsection{Continue of the proof}
For both cases it is true that:
\begin{lemma}
\begin{equation*}
\frac{1}{K+1} max \{ \| \tau \|^{L},  \| - \tau \|^{L} \} \leq \| \tau \| ^{N} \leq 2 \| \tau \|^{L} + \| - \tau \|^{L}
\end{equation*}
\end{lemma}
\begin{proof}
We will prove this for the generic  and the other case at the same time and we will refer to the corresponding definitions and propositions. Everything is true for both cases.\\
Firstly, we choose some candidates $\alpha$, $\beta$ which realise $\| \tau \|^{L}$, $\| - \tau \|^{L}$, respectively. Then we have that by \ref{Basic Prop} for any hyperbolic element $g$ in $\mathcal{O}$, $\frac{\tau(g)}{\ell(g)} \leq \frac{\tau(\alpha)}{\ell(\alpha)} = \| \tau \|^{L}$, $\|  $ and $\frac{-\tau(g)}{\ell(g)} \leq \frac{-\tau(\beta)}{\ell(\beta)} = \|  - \tau \|^{L}$. But since the minimum in \ref{1} (or \ref{2}), varies over the (projectively inequivalent) hyperbolic elements $g$ that realise $\ell(\alpha)$, we have that the for each $g$ in the sum we get: $\frac{\tau(g)}{\ell(\alpha)} \leq \frac{\tau(\alpha)}{\ell(\alpha)} = \| \tau \|^{L} $ and similarly $\frac{-\tau(g)}{\ell(\alpha)} \leq \frac{-\tau(\beta)}{\ell(\beta)} = \| - \tau \|^{L}$. Therefore we have that the positive summands in $N(\ell, \tau)$ are dominated by $\| - \tau \|^{L}$ and similarly the absolute value of negative summands are dominated by $\|  \tau \|^{L}$ and the right hand side follows.\\
Now the inequality $\frac{1}{K+1} \| \tau \|^{L} \leq \| \tau \| ^{N} $ is equivalent to $-N(\ell, \tau) \leq K \| \tau \|^{L}$, which is true using again the same argument as above.\\
Also, we have to prove that  $\frac{1}{K+1}  \| - \tau \|^{L} \leq \| \tau \|^{N}$, which is equivalent to the inequality 
\begin{equation*}
\| - \tau \|^{L} - N(\ell, \tau) \leq (K+1) \| \tau \|^{N}
\end{equation*}
Let $\alpha$ be a candidate that realises $\| -\tau \|^{L} $, then using the proposition \ref{lifts} (or \ref{realise}), we get that there is a term in $-N(\ell, \tau)$ of the form $\frac{\tau(\alpha)}{\ell(\alpha)}$ which is cancelled out with $\| -\tau \|^{L} = \frac{-\tau(\alpha)}{\ell(\alpha)}$. Finally, we apply again that each positive term of $-N(\ell,\tau)$ is dominated by $\| \tau \|^{L}$ and the left hand side inequality follows.
\end{proof}

Therefore $\| \cdot \|^{N}$ is a (non-symmetric) norm, just like $\| \cdot \|^{L}$ (positivity follows from the previous lemma, while subadditivity and multiplicativity with positive scalars are evident from the definition and the properties of $\| \cdot \|^{L}$).
As an immediate corollary, we can get a nice relation between $\| \tau \| ^{N}$ and $\| - \tau \| ^{N}$, which implies that the new norm is quasi-symmetric.
\begin{cor}\label{reverse}
There is a constant $A = 3(K +1 )$ so that
\begin{equation*}
\| \tau \|^{N} \leq A \| - \tau \|^{N}
\end{equation*}
\end{cor}

\begin{prop}\label{Relation}
If $\ell \in \Sigma_{T}$ and $ \tau \in T_{\ell} (\Sigma_{T})$ is integrable then
\begin{equation*}
\|  \tau \| ^{N} = \|  \tau \| ^{L} + d_{\tau} \Psi
\end{equation*}
where the third term is the derivative of $\Psi$ in the direction of $\tau$, i.e. the derivative from the right at $0$ of $t \rightarrow \Psi(\ell + t \tau)$.
\end{prop}
\begin{proof}
We prove it for the generic case. The proof for the other case is the same, but we just use the tree $T$ and we don't need to consider the covers.\\
We just need to prove that $d_{\tau} \Psi = \frac{1}{K+1}N(\ell, \tau)$. We use the chain rule and the Corollary \ref{derivative}.\\
Applying the Corollary \ref{derivative} on $T_i, \ell_i$ and $\tau_i$, we get that for any $[\alpha] \in H(T_i)$:
\begin{equation*}
d_{\tau _i} \ell_i ([\alpha]) = \tau_i (\alpha_i)
\end{equation*} 
where $a_i$ is the hyperbolic element that realises $\ell_i (\alpha)$ and on which $\tau_i$ is minimal. Therefore using the chain rule:

\begin{equation*}
d_{\tau_i} log \ell_i ([\alpha]) = \frac{\tau_i (\alpha _i)}{\ell_i (\alpha _i)} 
\end{equation*}
Therefore the resulting equations follows immediately, if we take the double sum.
\end{proof}

As in the free case, it is not difficult to see that we can extend the discussion to the whole Outer space $\mathcal{O}$. Firstly, we note that $\| \cdot \|^{L}, \| \cdot \|^{N} $ and $\Psi$ commute with the inclusions of simplices corresponding to collapsing forests without non-free vertices (since for the edges $e$ in that forest we have that $\ell (e) = \tau(e) = 0$).\\
Moreover, there is a permutation between the sets $H(R_{r,q})$ and $H(T)$ by construction and so we can identify  their homothety classes. Similarly, in the case with $n \geq 2$, we can identify the double covers of $T$ (with full stabilisers of vertices) with the double covers of $R_{r,n}$ (with full stabilisers of vertices), and the isomophism between their fundamental groups lifts to isomophisms between the fundamental groups  of their double covers with full stabiliers.\\
Therefore we can define $\Psi$ globally. Moreover, let  $\phi \in Out(G, \mathcal{O})$ be an automorphism and $T \in \mathcal{O}$ be an element of the outer space, then $\phi(T)$ is again an element with the same underlying tree as $T$ with the same metric, but with a different $G$-action. But the change of the action, only induces a permutation of the summands in the definition of $\Psi$ (for both cases).  As a consequence, $\Psi$ is $Out(G, \mathcal{O})$-invariant.

\section{Lenghts of Paths}
In the following sections all the ideas and the proofs are essentially the same as in  \cite{Algom Kfir}, however we include the most of the proofs for the convenience of the reader and for completeness.\\ 
Let $\gamma : [0,1] \rightarrow \mathcal{O}$ be a piecewise linear path. This means that $\gamma$ can be subdivided into finitely many subpaths so that each one is contained in some $C_j$ as in Corollary \ref{Closed Convex} on which $\Psi$ is smooth.\\
On the other hand, the Lipschitz length of $\gamma$ is
\begin{equation*}
len_{L} (\gamma) = \sup  \Big \{  \sum_{i=1}^{p} d(\gamma(t_{i-1}), \gamma(t_i))  : 0= t_0 < t_1 < ... < t_p = 1 \}
\end{equation*}

Suppose that $\Delta t_i= t_i - t_{i-1} $  is small. Since $\gamma (t_i)$ is $\ell$ and $\gamma '(t_i)$ is the vector of the tangent space $\tau$, we get:
\begin{equation*}
d(\gamma(t_{i-1}, \gamma(t_i)) = \frac{d(\gamma(t_{i-1}), \gamma(t_{i-1} + \Delta t_i) )}{\Delta t_i} \Delta t_i \sim \| ( \gamma(t_{i-1}), \gamma ' (t_{i-1}) ) \| ^{L} \Delta t_i
\end{equation*}

Thus
\begin{equation*}
len_{L} (\gamma) = \int_{0}^{1}  \| ( \gamma(t), \gamma ' (t) ) \| ^{L} dt
\end{equation*}

Similarly, we can also define  the length corresponding to the new norm.
\begin{equation*}
len_{N} (\gamma) = \int_{0}^{1}  \| ( \gamma(t), \gamma ' (t) ) \| ^{N} dt
\end{equation*}

\begin{prop}\label{len1}
Let $T,S \in \mathcal{O}$ and $\gamma : [0,1] \rightarrow \mathcal{O}$ be a path from $T$ to $S$ in $\mathcal{O}$.\\
Then $len_{N} (\gamma)  =  len_{L} (\gamma) + \Psi(S) - \Psi(T)$.
\end{prop}

\begin{proof}
We can use the Fundamental Theorem of Calculus to $\Psi \circ \gamma$, since $\Psi$ and $\gamma$ are piecewise differentiable. In order to simplify the notation, we write $\|  \gamma ' (t)  \| ^{\cdot}$ instead of $\| ( \gamma(t), \gamma ' (t) ) \| ^{\cdot}$  \\
Therefore:
\begin{equation*}
len_{N} (\gamma) = \int_{0}^{1}  \|  \gamma ' (t)  \| ^{N} dt 
\end{equation*}
On the other hand, combining it also with \ref{Relation}, we get that:
\begin{equation*}
\int_{0}^{1}  \|  \gamma ' (t)  \| ^{N} dt  =   \int_{0}^{1} [ \|  \gamma ' (t)  \| ^{L} + d_{\gamma' (t)} \Psi ]dt = len_{L} (\gamma) + \Psi(S) - \Psi(T)
\end{equation*}

\end{proof}

\begin{prop}\label{len2}
Let $T,S \in \mathcal{O}$ and $\gamma : [0,1] \rightarrow \mathcal{O}$ be a path from $T$ to $S$ in $\mathcal{O}$. Let $-\gamma : [0,1] \rightarrow \mathcal{O}$ be the reverse path $-\gamma (t) = \gamma(1-t)$. Then
\begin{equation*}
len_{N}(-\gamma) \leq A  len _{N}(\gamma)
\end{equation*} 
where $A$ is the constant from Corollary \ref{reverse}.
\end{prop}

\begin{proof}
Since $\gamma$ is piecewise $C^{1}$, for all but finitely many points $[-\gamma '] (s) = - \gamma'(1-s)$. Thus using the simplification of the notation as in the previous proof and changing the variable ($s = 1-t$), we get :
\begin{equation*}
len_{N}(-\gamma) = \int_{0}^{1}  \| [- \gamma '] (s)  \| ^{N} ds = \int_{0}^{1}  \| - \gamma ' (t)  \| ^{N} dt
\end{equation*}
But now we apply the Corollary \ref{reverse} and we have that:
\begin{equation*}
\int_{0}^{1}  \| - \gamma ' (t)  \| ^{N} dt \leq \int_{0}^{1} A  \| \gamma ' (t)  \| ^{N} dt = A len_{N} (\gamma)
\end{equation*}
\end{proof}

\section{Applications}
Now we are in position to prove the Main theorem and different applications. Let $A$ be the constant from Corollary \ref{reverse}.

\begin{cor}
For any $T \in \mathcal{O}$, for any $\phi \in Out(G, \mathcal{O})$ and any piecewise linear path $\gamma$ from $T$ to $\phi(T)$,
\begin{equation*}
len_{L}(\gamma) = len_N (\gamma)
\end{equation*}
Therefore
\begin{equation*}
len_L (\gamma) \leq A len_L (- \gamma)
\end{equation*}
\end{cor}
\begin{proof}
By Proposition \ref{len1}, we get that: \begin{equation*}
len_{N} (\gamma)  =  len_{L} (\gamma) + \Psi(\phi(T)) - \Psi(T).
\end{equation*}
But since as we have seen $\Psi$ is $Out(G, \mathcal{O})$ -invariant, which means $\Psi(\phi(T)) = \Psi(T)$, and therefore $len_{N} (\gamma) = len_{L} (\gamma)$. So using the Corollary \ref{reverse}, the result follows.
\end{proof}

\begin{theorem}
For any $T,S \in \mathcal{O}$ and for any piecewise linear path $\gamma$ from $T$ to $S$, fo \begin{equation*}
len_{N} (\gamma) \leq A len_{N} (- \gamma) +(A+1)[\Psi(T) - \Psi(S)] 
\end{equation*}
where $A$ is the constant of \ref{reverse}.
\end{theorem}
\begin{proof}
Combining the Propositions \ref{len1} and \ref{len2}: 
\begin{equation*}
  len_{L} (\gamma) + \Psi(S) - \Psi(T) = len_{N} (\gamma) 
\end{equation*}
\begin{equation*}
  \leq A len_N (- \gamma) = A len_L (- \gamma) + A[\Psi(T) - \Psi(S)].
\end{equation*}
Therefore we get the requested result for $len_{L} (\gamma)$.
\end{proof}

\begin{thm}\label{Main theorem}
For any $T, S \in \mathcal{O}$,
\begin{equation*}
d(T,S) \leq A d(S,T) + (A+1)[\Psi(T) - \Psi(S)]
\end{equation*}
\end{thm}
\begin{proof}
Let $T,S \in \mathcal{O}$ and let's choose a piecewise linear geodesic path from $S$ to $T$ which we denote by $ - \gamma$.
We apply the previous theorem to $\gamma$, which is a path from $T$ to $S$.
\end{proof}
\begin{remark}
Note here that since $\phi$ is $Out(G,\mathcal{O})$-invariant, if $T,S \in \mathcal{O}$ are in the same orbit, then $d(T,S) \leq A d(S,T)$
\end{remark}
Now we prove a theorem about the relation between the expansion factors of an IWIP relative to $\mathcal{O}$ and its inverse. This is a generalisation of the theorem of Handel and Mosher in \cite{Handel + Mosher}, about the relation of the expansions factors of an IWIP automorphism of a free group and its inverse.
\begin{theorem}
For any IWIP automorphism $\phi \in Out(G, \mathcal{O})$ relative to $\mathcal{O}$, let $\lambda$ be the expansion factor of $\phi$ and $\mu$ be the expansion factor of the IWIP $\phi ^{-1}$. Then $\mu \leq \lambda ^{A}$, where $A$ as above.
\end{theorem}

\begin{proof}
Let $f : T \rightarrow T $ be an optimal train track representative of $\phi$ and $h : S \rightarrow S $ be an optimal train track representative of $\phi ^{-1}$, which means that $d(\phi^{k}(T), T) = k \log \lambda$ and $d(\phi^{-k}(S), S) = k \log \mu $, for every natural number $k$ .\\
Let choose a number $D \geq \max \{d(T,S), d(S,T) \}$ and then, by the triangle inequality, we get that for any natural number $k$, $d(\phi^{k}(T), T) \geq d(\phi^{k}(S), S) - d(\phi^{k}(S), \phi^{k}(T))- d(T,S) \geq klog \mu  -2D$. On the other hand, using the Main Theorem, $d(\phi^{k}(T),T  ) \leq A \cdot d(T, \phi^{k}(T)) = A \cdot k log \lambda$. Therefore combining the inequalities we get
\begin{equation*}
A \cdot k \cdot log \lambda \geq k \cdot log \mu -2D
\end{equation*}
As consequence, for every $k$ we have that $log \mu \leq A \cdot log \lambda +\frac{2D}{k}$ and sending $k$ to infinity, we get $\frac{log \mu}{log \lambda} \geq A$ or equivalently $\mu \leq \lambda ^{A}$
\end{proof}
 However, Handel and Mosher proved also a more general theorem for automorphisms of free groups, and more specifically they proved a relation between the sets of expansions factors of any automorphism and its inverse, using the notion of strata of relative train tracks representatives and the powerful machinery of laminations of Bestvina, Feighn and Handel (see for example \cite{BFMH1}). Using the theorem above for the general case, we can get as a corollary a special case of this theorem.\\
If $\phi \in Out(F_n)$ and $f: T \rightarrow T$ is a relative train track representative of $\phi$, denoting by $\lambda$ the expansion factor of the top stratum, then by a remark of \cite{FM} there is a relative outer space $\mathcal{O}$ on which $\phi \in Out(G, \mathcal{O})$ and $\phi$ is irreducible relative to $\mathcal{O}$ (equivalently, a maximal free factor system). Moreover, the same is true for $\phi ^{-1}$ using the same space (equivalently the same free factor system). We distinguish two cases. If $\lambda = 1$, then $\phi$ and $\phi ^{-1}$ fix some point of $\mathcal{O}$. Which means that there is a relative train track representative of $\phi^{-1}$, for which the expansion factor of the top stratum is $1$. If $\lambda >1$, then $\phi, \phi ^{-1}$ are IWIP relative to $\mathcal{O}$, and let's denote by $\mu >1$ the expansion factor of $\phi ^{-1}$ relative to $\mathcal{O}$, which means that there is a relative train track representative $h$ of $\phi^{-1}$ with the expansion factor of the top stratum to be $\mu$. Using the theorem above, we get that $\log \mu$ and $\log \lambda$ are comparable and the constant depends on the group. Note that, in general, we don't have the uniqueness of the maximal free factor system, however using this approach we can get a correspondence between the maximal free factor systems of $\phi$ and $\phi ^{-1}$, and in particular their relative expansion factors.


Let $\epsilon$ be a positive number. Let denote by $\mathcal{O} _{\geq{\epsilon}}$ the thick part of $\mathcal{O}$, i.e. the set of all trees of $\mathcal{O}$, which don't contain hyperbolic elements shorter than $\epsilon$,  then it is co-compact for every $\epsilon$.

\begin{theorem}
For every $\epsilon > 0$ there is a constant $B$ so that for every $T,S \in \mathcal{O}_{\geq \epsilon}$ and any piecewise linear path $\gamma$ from $T$ to $S$:
\begin{equation*}
\frac{1}{A} len(\gamma) - B \leq len(- \gamma) \leq Alen(\gamma) + B
\end{equation*}
Moreover, there is a constant $D$ such that for all $T, S \in \mathcal{O} _{\geq \epsilon}$
\begin{equation*}
d(S,T) \leq D d(T, S)
\end{equation*}
\end{theorem}
The proof is exactly the same in the free case using the fact the $\epsilon$-thick part of $\mathcal{O}$ is co-compact.

\end{document}